\documentclass[12pt, a4]{amsart}
\usepackage{amsmath, amsthm}
\usepackage[psamsfonts]{amssymb}
\usepackage{color}
\usepackage{epsfig}
\usepackage[colorlinks]{hyperref}
\usepackage{pgf,tikz}
\usepackage{fullpage}

\newtheorem{theorem}{Theorem}[section]

\newtheorem{corollary}[theorem]{Corollary}

\theoremstyle{definition}

\newtheorem{example}[theorem]{Example}
\newtheorem{remark}[theorem]{Remark}
\theoremstyle{remark}
\usepackage{graphicx}
\usepackage{epsfig}

\newcommand{\R}{{\mathbb R}}

\begin{document}

\title[On the Liebau phenomenon]{A topological approach to periodic oscillations related to the Liebau phenomenon}


\author{J. A. Cid}
\address{Jos{\'e} {\'A}ngel Cid, Departamento de Matem\'aticas, Universidade de Vigo, 32004, Pabell\'on 3, Campus de Ourense, Spain}%
\email{angelcid@uvigo.es}%

\author{G. Infante}
\address{Gennaro Infante, Dipartimento di Matematica e Informatica, Universit\`{a} della
Calabria, 87036 Arcavacata di Rende, Cosenza, Italy}%
\email{gennaro.infante@unical.it}%

\author{M. Tvrd\'y}
\address{Milan Tvrd\'y, Mathematical Institute, Academy of Sciences
of the Czech Republic, CZ~115~67 Praha~1, \v{Z}itn\'{a}~25, Czech Republic}%
\email{tvrdy@math.cas.cz}%

\author{M. Zima}
\address{Miros\l awa Zima, Department of Differential Equations and Statistics,
Faculty of Mathematics and Natural Sciences, University of Rzesz\'ow,
Pigonia 1, 35-959 Rzesz{\'o}w, Poland.}%
\email{mzima@ur.edu.pl}%

\begin{abstract}
We give some sufficient conditions for existence, non-existence and localization of positive solutions for a periodic boundary value problem related to the Liebau phenomenon. Our approach is of topological nature and relies on the Krasnosel'ski\u\i{}-Guo theorem on cone expansion and compression. Our results improve and complement earlier ones in the literature.
\end{abstract}

\maketitle

\section{Introduction}

In the 1950's the physician G. Liebau developed some experiments dealing with a valveless pumping phenomenon arising on blood circulation and that has been known for a long time: roughly speaking, Liebau showed experimentally that a periodic compression made on an asymmetric part of a fluid-mechanical model could produce the circulation of the fluid without the necessity of a valve to ensure a preferential direction of the flow \cite{borzipropst,liebau,moser}. After his pioneering work this effect has been known as the Liebau phenomenon.

In \cite{propst} G. Propst, with the aim of contributing to the theoretical understanding of the Liebau phenomenon, presented some differential equations modeling a periodically forced flow through different pipe-tank configurations. He was able to prove the presence of pumping effects in some of them, but the apparently simplest model, the ``one pipe-one tank" configuration, skipped his efforts due to a singularity in the corresponding differential equation model, namely
\begin{equation}
  \label{eqprobsing}
  \left\{
\begin{array}{l}
   u''(t)+a\,u'(t)=\displaystyle \frac 1{u}\left(e(t)-b\,(u'(t))^2\right)-c, \quad   t\in [0,T],  \\
 u(0)=u(T),\quad u'(0)=u'(T),
\end{array}
\right.
\end{equation}
being $a\ge 0,\, b>1,\, c>0 \, \mbox{and} \,
\,e(t) \, \mbox{continuous and $T$-periodic on\ }\R.$

The singular periodic problem \eqref{eqprobsing} was studied in  \cite{CiPrTv}, where the authors gave general results for the existence and asymptotic stability of positive solutions by performing the change of variables $u\,{=}\,x^{\mu}$, where $\mu\,{=}\,\frac 1{b\,{+}\,1},$
which transforms the singular problem \eqref{eqprobsing} into the regular one
\begin{equation}
  \label{eqprobreg}
  \left\{
\begin{array}{l}
  x''(t)+a\,x'(t)=\displaystyle \frac{e(t)}{\mu} \,x^{1\,{-}\,2\,\mu}(t)-\displaystyle\frac{c}{\mu}\,x^{1\,{-}\,\mu}(t), \quad   t\in [0,T],  \\
 x(0)=x(T),\quad x'(0)=x'(T).
\end{array}
\right.
\end{equation}
Then the existence and stability of positive solutions for  \eqref{eqprobreg} were obtained by means of the lower and upper solution technique \cite{deCosHa} and using tricks analogous to those used in \cite{RaTvVr}.

In this paper we deal with the existence of positive solutions for the following generalization of the problem \eqref{eqprobreg}
\begin{equation}
  \label{eqprob}
  \left\{
\begin{array}{l}
    x''(t)+ax'(t)=r(t)x^{\alpha}(t)-s(t)x^{\beta}(t), \quad   t\in [0,T],  \\
  x(0)=x(T), \quad   x'(0)=x'(T),
\end{array}
\right.
\end{equation}
 where we assume
 \begin{itemize}
\item[(H0)] $a\ge 0$, $r,s\in {\mathcal C}[0,T]$, $0<\alpha<\beta<1.$
\end{itemize}
Note that, by defining $r(t)=\displaystyle\frac{e(t)}{\mu},\, s(t)=\displaystyle\frac{c}{\mu},\,
  \alpha=1\,{-}\,2\,\mu$ and  $\beta=1\,{-}\,\mu$, the problem \eqref{eqprobreg} fits within \eqref{eqprob}.

Our approach is essentially of topological nature: in Section \ref{prel} we rewrite problem \eqref{eqprob} as an equivalent fixed point problem suitable to be treated by means of the Krasnosel'ski\u\i{}-Guo cone expansion/compression fixed point  theorem. A careful analysis of the related Green's function, necessary in our approach, is postponed to a final Appendix. In Section \ref{main} we present our main results: existence, non-existence and localization criteria for positive solutions of the problem \eqref{eqprob}. Some corollaries with more ready-to-use results are also addressed. We point out that our results are valid not only for the more general problem \eqref{eqprob}, but also when applied to the singular model problem \eqref{eqprobsing} we improve previous results of \cite{CiPrTv}.

\section{A fixed point formulation}\label{prel}

First of all, by means of a shifting argument, we rewrite the problem \eqref{eqprob} in the equivalent form
\begin{equation}
  \label{eqprob+m}
  \left\{
\begin{array}{l}
     x''(t)+ax'(t)+m^2x(t)=r(t)x^{\alpha}(t)-s(t)x^{\beta}(t)+m^2x(t):=f_m(t,x(t)), \ t\in [0,T],  \\
  x(0)=x(T), \quad   x'(0)=x'(T),
\end{array}
\right.
\end{equation}
with  $m\in \R$. A similar approach has been used, under a variety of boundary conditions in \cite{Han, gippat, Tor, WebbZima}.

We say that problem \eqref{eqprob+m} is non-resonant if zero is the unique solution of the homogeneous linear problem
$$ \left\{
\begin{array}{l}
     x''(t)+ax'(t)+m^2x(t)=0, \quad   t\in [0,T],  \\
  x(0)=x(T), \quad   x'(0)=x'(T).
\end{array}
\right.
$$
In this case the non-homogeneous linear problem
\begin{equation}\label{eqprobhom}
  \left\{
\begin{array}{l}
     x''(t)+ax'(t)+m^2x(t)=h(t), \quad   t\in [0,T],  \\
  x(0)=x(T), \quad   x'(0)=x'(T),
\end{array}
\right.
\end{equation}
is also uniquely solvable and its unique solution is given by
$$
  {\mathcal K}h(t)=\int_0^T G_m(t,s)h(s)ds,
$$
where $G_m(t,s)$ is the related Green's function (see \cite{Cab,CaCiMa}).

In particular, if $m>0$ and $m^2< \left(\displaystyle\frac{\pi}{T}\right)^2+\left(\displaystyle\frac{a}{2}\right)^2$ then the problem \eqref{eqprob+m} is non-resonant and moreover $G_m(t,s)$ satisfies the following properties (see the Appendix for the details):
\begin{itemize}

\item[i)] $G_m(t,s)>0,$ for all $(t,s)\in [0,T]\times [0,T]$.

\item[ii)] $\displaystyle\int_0^T G_m(t,s)ds=\frac{1}{m^2}.$

\item[iii)] There exists a constant $c_m\in (0,1)$ such that $$G_m(t,s)\ge G_m(s,s)\ge c_m G_m(t,s),\ \text{for all}\ (t,s)\in [0,T]\times [0,T].$$
\end{itemize}

A \textit{cone} in a Banach space $X$ is a closed,
convex subset of $X$ such that $\lambda \, x\in K$ for $x \in K$ and
$\lambda\geq 0$ and $K\cap (-K)=\{0\}$. Here we work in the space $X={\mathcal C}[0,T]$ endowed with the usual maximum norm
$\|x\|=\max\{|x(t)|:t\in[0,T]\}$ and use the  cone
$$
P=\{x\in X: x(t)\ge c_m \|x\|\ \text{on}\ [0,T]\},
$$
a type of cone firstly used by Krasnosel'ski\u\i, see for example \cite{krzab}, and D.~Guo, see for example \cite{GuoLak}. The cone $P$ is particularly useful when dealing with singular nonlinearities (see \cite{klamc04}), or for localizing the solutions (see for example \cite{gippmt}).

Let us define the operator $F:P\to X$
\begin{equation}\label{F-def}
Fx(t)=\int_0^T G_m(t,s)f_m(s,x(s))ds.
\end{equation}
Note that a fixed point of $F$ in $P$ is a non-negative solution of the problem \eqref{eqprob+m}. In order to get such a fixed point we employ the following well-known Krasnosel'ski\u\i{}-Guo cone compression/expansion theorem.
\begin{theorem}\label{krasno} \cite{GuoLak}
Let $P$ be a cone in $\,X$ and suppose
that $\,\Omega_1$ and $\Omega_2$ are bounded open sets in $\,X$ such
that $\,0 \in \Omega_1$ and $\ \overline{\Omega}_1\subset\Omega_2$.
Let $F:P\cap(\overline{\Omega}_2 \setminus \Omega_1)\to P$ be a
completely continuous operator such that one of the following
conditions holds:
\begin{itemize}
\item[{(i)}]$\|Fx\|\ge \|x\|\,$ for $x\in P\cap\partial
\Omega_1$ and $\,\|Fx\|\le\|x\|\,$ for $x\in P\cap\partial\Omega_2$,
\item[{(ii)}] $\|Fx\|\le\|x\|\,$ for $x\in P\cap\partial\Omega_1$
and $\,\|Fx\|\ge \|x\|\,$ for $x\in P \cap\partial
\Omega_2.$
\end{itemize}
Then $F$ has a fixed point in the set
$P\cap(\overline{\Omega}_2\setminus \Omega_1)$.
\end{theorem}

In the sequel, for a given continuous function $h{:}\,[0,T]\to\R$ we use the notation
$$
h_*=\min\{h(t):t\in[0,T]\}
  \quad \mbox{and} \quad h^*=\max\{h(t):t\in[0,T]\}.
$$

\section{Main results}\label{main}

Integrating on $[0,T]$ the differential equation in the problem \eqref{eqprob}, and taking into account the boundary conditions, we get a necessary condition for the existence of positive solutions, namely
$$0=\int_0^T [ r(t)x^{\alpha}(t)-s(t)x^{\beta}(t)] dt.$$
So it is easy to arrive to the following non-existence result.
\begin{theorem} If one of the following conditions holds,
\begin{enumerate}
\item $r_*\ge 0$ and $s^*< 0$,
\item $r_*> 0$ and $s^*\le 0$,
\item $r^*\le 0$ and $s_*> 0$,
\item $r^*< 0$ and $s_*\ge 0$,
\end{enumerate}
 then problem \eqref{eqprob} does not have positive solutions.
\end{theorem}

Our main result, concerning not only the existence but also the localization of positive solutions, is the following one.
\begin{theorem}\label{thmain} Assume that {\rm (H0)} and the following condition hold:
\begin{itemize}
\item[(H1)] There exist $m>0$ and $0<R_1<R_2$ such that $m^2< \left(\displaystyle\frac{\pi}{T}\right)^2+\left(\displaystyle\frac{a}{2}\right)^2$,
 \begin{equation}\label{eq1R1R2}  f_m(t,x)=r(t)x^{\alpha}-s(t)x^{\beta}+m^2x\ge 0 \quad \mbox{for $t\in[0,T]$ and $x\in [c_m R_1,R_2]$,}\end{equation}
 \begin{equation}\label{eq2R1R2}
f_m(t,x)\ge m^2R_1 \quad \mbox{for $t\in[0,T]$ and $x\in [c_m R_1,R_1]$}
\end{equation}
and
 \begin{equation}\label{eq3R1R2}
f_m(t,x)\le m^2 R_2 \quad \mbox{for $t\in[0,T]$ and $x\in [c_m R_2,R_2]$}.
\end{equation}

\end{itemize}
Then, the problem \eqref{eqprob} has a positive solution $x(t)$ such that
$$c_m R_1\le x(t)\le R_2.$$
\end{theorem}

\begin{proof} The problem  \eqref{eqprob+m} with the value of $m$ given by (H1) is non-resonant and its Green's function $G_m(t,s)$ satisfies the properties i), ii) and iii) stated in the Introduction. Let us define
$$\Omega_j=\{x\in X: \|x\|<R_j\}\quad j=1,2.$$
Then by \eqref{eq1R1R2} we have $f_m(t,x(t))\ge 0$ for $x \in P\cap(\overline{\Omega}_2 \setminus \Omega_1)$.
Now, for every $x \in P\cap(\overline{\Omega}_2 \setminus \Omega_1)$ and for every $t\in [0,T]$, we have that
$$\int_0^T G_m(s,s) f_m(s,x(s))ds \le Fx(t)= \int_0^T G_m(t,s)f_m(s,x(s))ds\le \int_0^T \frac{G_m(s,s)}{c_m} f_m(s,x(s))ds,$$
and therefore $F(P\cap(\overline{\Omega}_2 \setminus \Omega_1))\subset P$. A standard argument shows that $F$ is a completely continuous operator.

Now, we check that condition (i) in Theorem  \ref{krasno} is fulfilled.

\medbreak
\noindent {\it Claim 1. $\|Fx\|\ge \|x\|$ for $x\in P\cap \partial \Omega_1$.}

Let $x\in P\cap \partial \Omega_1$, that is, $x\in P$ and $\|x\|=R_1$. Then we have that $c_m R_1\le x(t)\le R_1$ for all $t\in [0,T]$ and hence
\begin{eqnarray*}
Fx(t)&=& \int_0^T G_m(t,s)f_m(s,x(s))ds\ge \int_0^T G_m(t,s) m^2R_1ds
 \\ &=& m^2R_1 \int_0^T G_m(t,s) ds=m^2R_1 \frac{1}{m^2}=R_1=\|x\|.
\end{eqnarray*}

 \medbreak
\noindent {\it Claim 2. $\|Fx\|\le \|x\|$ for $x\in P\cap \partial \Omega_2$.}

Let $x\in P\cap \partial \Omega_2$, that is, $x\in P$ and $\|x\|=R_2$. Then $c_m R_2\le x(t)\le R_2$ for all $t\in [0,T]$ and hence
\begin{eqnarray*}
Fx(t)&=& \int_0^T G_m(t,s)f_m(s,x(s))ds\le \int_0^T G_m(t,s) m^2 R_2ds
 \\ &=& m^2 R_2 \int_0^T G_m(t,s) ds=m^2 R_2 \frac{1}{m^2}=R_2=\|x\|.
\end{eqnarray*}

\medbreak

Thefore by Theorem \ref{krasno} the existence of a solution of the problem \eqref{eqprob} with the desired localization property immediately follows.
\end{proof}

The following result is a consequence of Theorem \ref{thmain}.
\begin{theorem}\label{thmain2} Assume that {\rm (H0)} and the following conditions hold:
\begin{itemize}
\item[(H2)] There exists $m>0$ such that $m^2< \left(\displaystyle\frac{\pi}{T}\right)^2+\left(\displaystyle\frac{a}{2}\right)^2$ and
$$f_m(t,x)=r(t)x^{\alpha}-s(t)x^{\beta}+m^2x\ge 0 \quad \mbox{for $t\in[0,T]$ and $x\ge 0$.}$$
\item[(H3)] $r_*>0$ and $s_*>0$.
\end{itemize}

Then the problem \eqref{eqprob} has a positive solution.
\end{theorem}
\begin{proof} We are going to prove that assumption (H1) in Theorem \ref{thmain} is satisfied. By (H0), (H3) and since $c_m\in (0,1)$ we can choose $0<R_1<R_2$ such that
\begin{equation}\label{eqR1} (1-c_m)m^2 R_1^{1-\alpha}+s^*R_1^{\beta-\alpha}\le r_*c_m^{\alpha}	
\end{equation}
 and
 \begin{equation}\label{eqR2}  r^*\le s_* c_m^{\beta} R_2^{\beta-\alpha}. \end{equation}
Note that the inequality \eqref{eq1R1R2} follows from (H2) so it only remains to check the inequalities \eqref{eq2R1R2} and \eqref{eq3R1R2}.

\medbreak

If $t\in[0,T]$ and $x\in [c_m R_1,R_1]$ then, taking into account (H3) and  \eqref{eqR1}, we have that
$$f_m(t,x)\ge r_*(c_mR_1)^{\alpha}-s^*R_1^{\beta}+m^2c_mR_1\ge m^2R_1,$$
so the inequality \eqref{eq2R1R2} is fulfilled.

 \medbreak

If $t\in[0,T]$ and $x\in [c_m R_2,R_2]$ then, taking into account (H3) and  \eqref{eqR2}, we have that
$$f_m(t,x)\le r^* R_2^{\alpha}-s_*(c_m R_2)^{\beta}+m^2 R_2\le m^2 R_2,$$
so  the inequality \eqref{eq3R1R2} is also fulfilled.
\end{proof}

\begin{remark} Note that (H0) and (H2) imply that $r_*>0$, thus the first part of condition (H3) is redundant but we have included it for the sake of clarity. On the other hand, notice that condition \eqref{eq1R1R2} in Theorem \ref{thmain} could be satisfied even if $r(t)$ assumes negative values.
\end{remark}

Next, we present an explicit condition sufficient in order to get the condition (H2) in Theorem \ref{thmain2}.
\begin{corollary}\label{cor1} Assume that {\rm (H0)} and {\rm (H3)} hold and, moreover, that
\begin{equation}\label{eqexplcond} s^*<\min\{ \left(\displaystyle\frac{\pi}{T}\right)^2+\left(\displaystyle\frac{a}{2}\right)^2, r_*\}.
\end{equation}
Then the problem \eqref{eqprob} has a positive solution.
\end{corollary}

\begin{proof} It is enough to show that condition \eqref{eqexplcond} implies (H2). Indeed, by \eqref{eqexplcond} we can choose $m>0$ such that
$$s^*<m^2<\left(\displaystyle\frac{\pi}{T}\right)^2+\left(\displaystyle\frac{a}{2}\right)^2.$$

Now, we are going to prove that for such $m$ we have that $f_m(t,x)\ge 0$ for all $t\in [0,T]$ and all $x\ge 0$.
\medbreak
\noindent {\it Step 1. We show that $f_m(t,x)\ge 0$ for all $t\in [0,T]$ and all $x\in [0,1]$.}

Since $0<\alpha<\beta<1$ we have $x^{\beta} \le x^{\alpha}$ for $0\le x\le 1$.  Moreover $s^*<r_*$ by \eqref{eqexplcond} and then
\begin{eqnarray*}
f_m (t,x) &=& r(t)x^{\alpha}-s(t)x^{\beta}+m^2 x \ge r_* x^{\alpha}-s^* x^{\beta}+m^2 x \\
 &\ge& r_* x^{\beta}-s^* x^{\beta}+m^2 x=(r_* -s^*) x^{\beta}+m^2 x> 0.
\end{eqnarray*}

\medbreak
\noindent {\it Step 2. We show that $f_m(t,x)\ge 0$ for all $t\in [0,T]$ and all $x\ge 1$.}

Since $0<\beta<1$ we have $x^{\beta} \le x$ for $x\ge 1$ and therefore
\begin{eqnarray*}
f_m (t,x) &=& r(t)x^{\alpha}-s(t)x^{\beta}+m^2 x \ge r_* x^{\alpha}-s^* x^{\beta}+m^2 x \\
 &\ge& r_* x^{\alpha}-s^* x^{\beta}+m^2 x^{\beta}=r_* x^{\alpha}+(m^2-s^*) x^{\beta} \ge 0.
\end{eqnarray*}
\end{proof}

In the particular case of problem \eqref{eqprobreg} we recover \cite[Theorem 1.8]{CiPrTv}.
\begin{corollary}\label{corexist}
Assume that $a\ge 0$, $b>1$, $c>0$ and $e_*\,{>}\,0.$ Then problem \eqref{eqprobreg} has a positive solution provided that
\begin{equation}\label{eqexistprobreg}
\frac{(b\,{+}\,1)\,c^2}{4\,e_*}<\left(\frac{\pi}{T}\right)^2{+}\,\frac{a^2}4\,.
\end{equation}
\end{corollary}

\begin{proof} We recall that the problem \eqref{eqprobreg} is of the form \eqref{eqprob} with
$$\mu=\frac{1}{b+1}, \, r(t)=\frac{e(t)}{\mu},\, s(t)=\frac{c}{\mu},\,
  \alpha=1\,{-}\,2\,\mu,\,\, \beta=1\,{-}\,\mu.$$
From our assumptions it follows immediately that the conditions (H0) and (H3) of Theorem \ref{thmain2} hold. Thus it only remains to check the condition (H2). By  \eqref{eqexistprobreg} we can choose $m>0$ such that
\begin{equation}\label{eqexistprobreg2}
\frac{(b\,{+}\,1)\,c^2}{4\,e_*}<m^2<\left(\frac{\pi}{T}\right)^2{+}\,\frac{a^2}4\,.
\end{equation}

Then for all $t\in [0,T]$ and $x\ge 0$ we have
\begin{eqnarray*}
f_m(t,x)&=&  \displaystyle \frac{e(t)}{\mu} \,x^{1\,{-}\,2\,\mu}-\displaystyle\frac{c}{\mu}\,x^{1\,{-}\,\mu}+m^2 x\ge  \displaystyle \frac{e_*}{\mu} \,x^{1\,{-}\,2\,\mu}-\displaystyle\frac{c}{\mu}\,x^{1\,{-}\,\mu}+m^2 x \\
&=& \frac{x^{1\,{-}\,2\,\mu}}{\mu}  \left( e_* -c x^{\mu}+m^2 \, \mu\, x^{2 \mu}Ê\right)\ge 0,
\end{eqnarray*}
because $ e_* -c x^{\mu}+m^2 \, \mu\, x^{2 \mu}>0Ê$ (it is a parabola in the variable $x^{\mu}$ which has by \eqref{eqexistprobreg2} negative discriminant
$c^2 - 4 m^2 \mu e_*<0$, so it has constant sign, and positive independent term $e_*>0$).

\end{proof}

The following example shows that Theorem  \ref{thmain} is in fact more general than Corollary \ref{corexist}. Therefore,
even in the case of the model problem  \eqref{eqprobreg}, Theorem  \ref{thmain} is a true extension of \cite[Theorem 1.8]{CiPrTv}. Moreover, notice that we also obtain information about the localization of the solution which was not provided in \cite{CiPrTv}. Some computations here were made with MAPLE.
\begin{example}\label{ex} Consider the problem \eqref{eqprobreg}, that is,
$$
  \left\{
\begin{array}{l}
  x''(t)+a\,x'(t)=\displaystyle \frac{e(t)}{\mu} \,x^{1\,{-}\,2\,\mu}-\displaystyle\frac{c}{\mu}\,x^{1\,{-}\,\mu}, \quad   t\in [0,T],  \\
 x(0)=x(T),\quad x'(0)=x'(T),
\end{array}
\right.
$$
with the parameter values $a=1.6$, $e(t)\equiv 1.54$, $\mu=0.01$, $c=1.49$ and $T=1$.

Therefore we have
 $$b=99 \, \mbox{(recall that $\mu:=\frac{1}{b+1})$}, \, r(t):=\frac{e(t)}{\mu}\equiv 154, \, s(t):=\frac{c}{\mu}\equiv 149,$$ $$ \alpha:=1-2\mu=0.98 \quad \mbox{and}\quad  \beta:=1-\mu=0.99.$$
Thus, condition \eqref{eqexistprobreg} in Corollary \ref{corexist} is not satisfied because
 $$\frac{(b\,{+}\,1)\,c^2}{4\,e_*}\approx 36.0406>10.5096\approx\left(\frac{\pi}{T}\right)^2{+}\,\frac{a^2}4.$$
 Notice that neither condition \eqref{eqexplcond} in Corollary \ref{cor1} is satisfied since
 $$s^*=149>10.5096\approx\left(\frac{\pi}{T}\right)^2{+}\,\frac{a^2}4.$$

However, the condition (H0) in Theorem  \ref{thmain} is clearly satisfied and if we define $m:=0.7$, $R_1:=27$ and $R_2:=29$ then (H1) is also fulfilled  (notice that $c_m\approx 0.9414$ as it is given by \eqref{c2}).

\vspace*{0.25cm}
\begin{minipage}[b]{0.5\linewidth}

\includegraphics[scale=0.77]{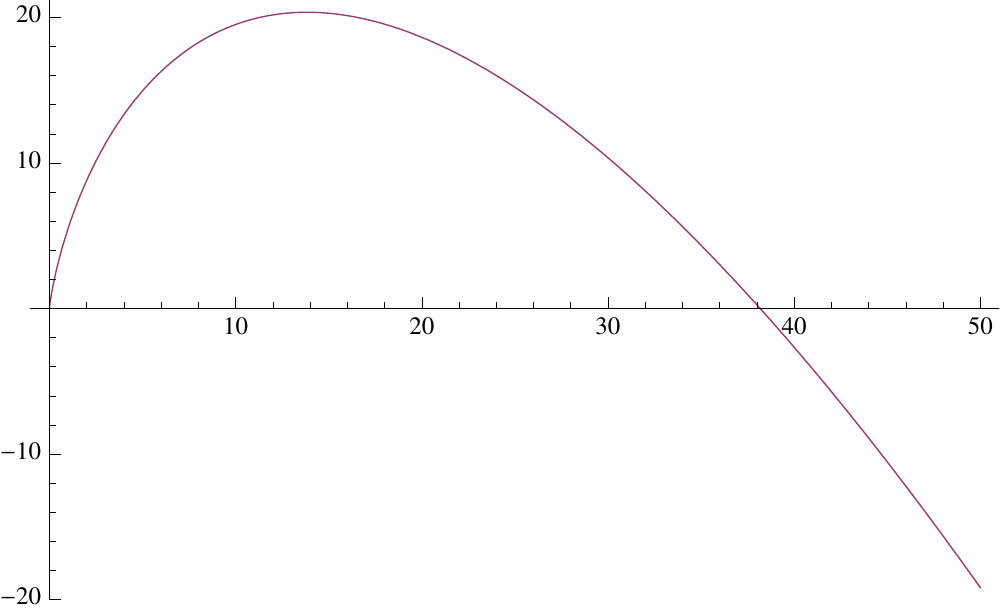}

\begin{center}
$f_m(t,x)$ on $[0,50]$
\end{center}

\end{minipage} \hfill
\begin{minipage}[b]{0.5\linewidth}

\includegraphics[scale=0.77]{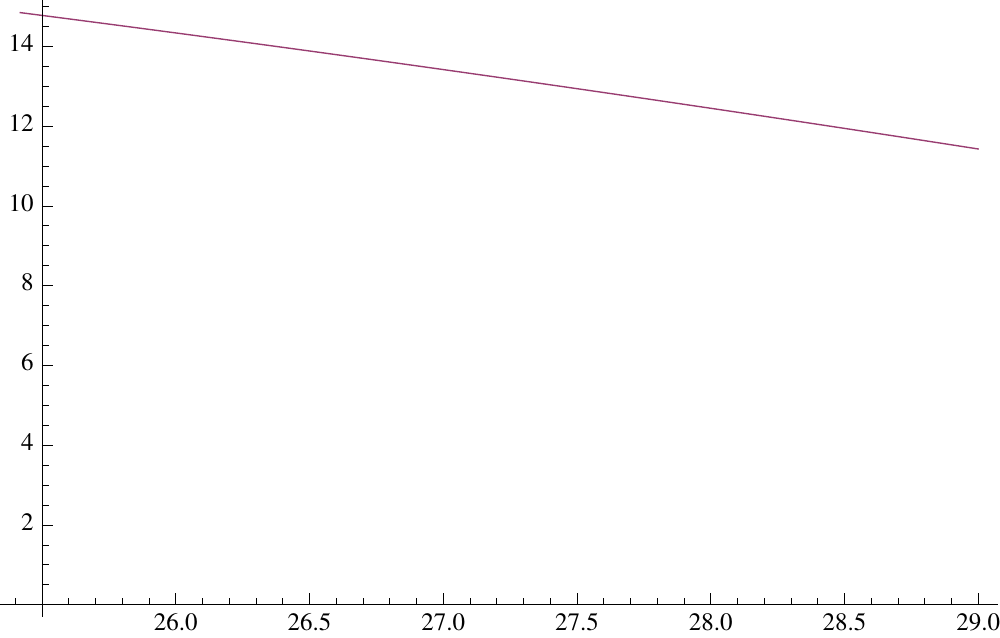}

\begin{center}
$f_m(t,x)$ on $[c_m R_1,R_2]$
\end{center}

\end{minipage}

\vspace*{0.5cm}

\begin{minipage}[b]{0.5\linewidth}

 \includegraphics[scale=0.77]{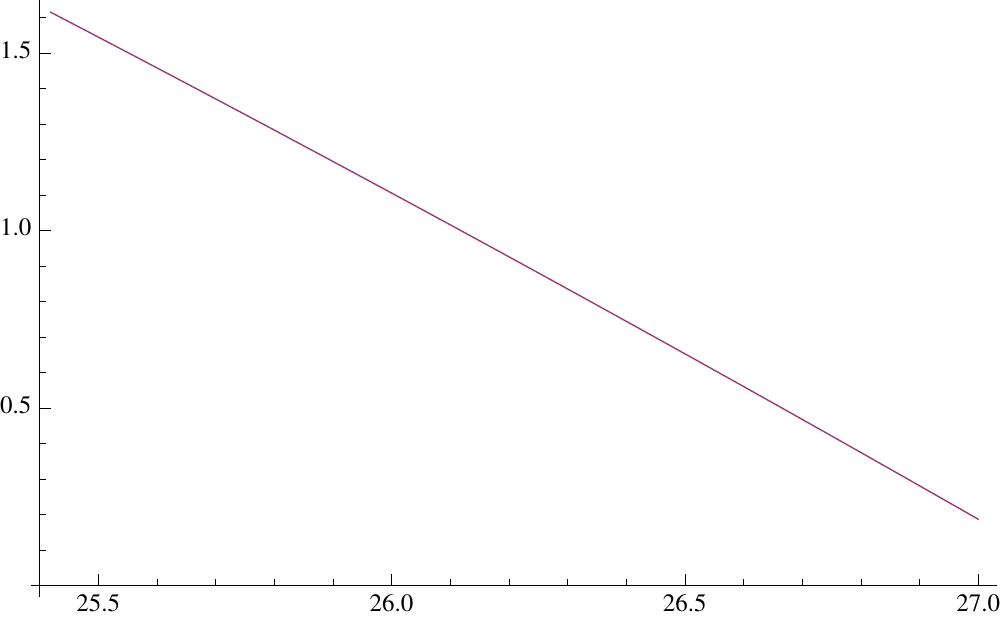}

\begin{center}
$f_m(t,x)-m^2R_1$ on $[c_m R_1,R_1]$
\end{center}

\end{minipage}
\begin{minipage}[b]{0.5\linewidth}

  \includegraphics[scale=0.77]{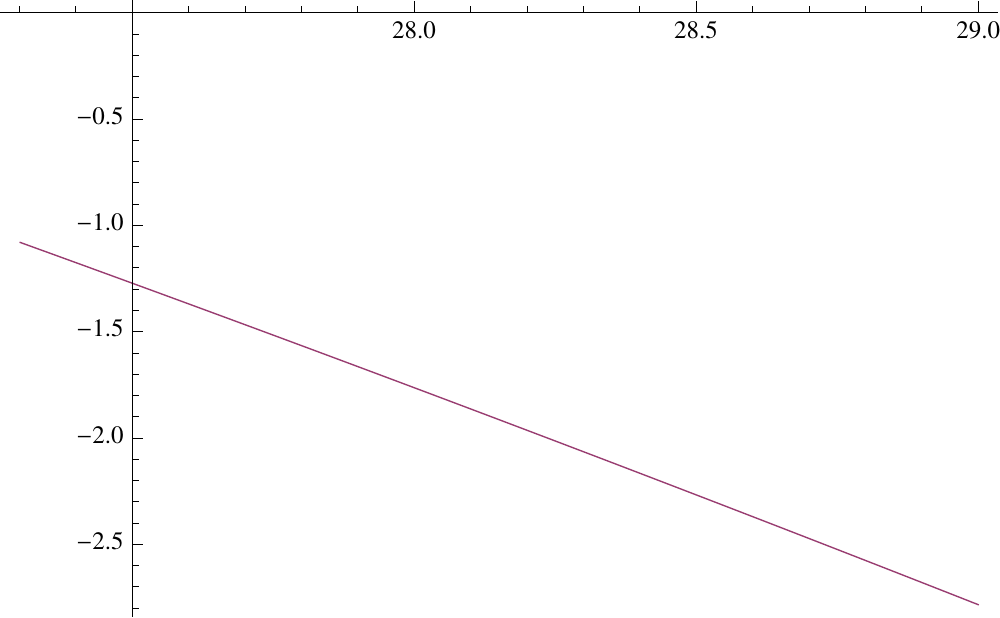}

 \begin{center}
$f_m(t,x)-m^2R_2$ on $[c_m R_2,R_2]$
\end{center}

\end{minipage} \vspace{0.75cm}

Thus, Theorem  \ref{thmain} ensures the existence of a solution $x(t)$ of this problem such that
$$25.4189\approx c_m R_1\le x(t)\le R_2=29.$$
Observe that $f_m$ is not positive for all $x\ge 0$.

\end{example}

\begin{remark} From a careful reading of the proof of \cite[Theorem 1.8]{CiPrTv} it follows that the existence of a positive periodic solution to problem \eqref{eqprobreg}  is ensured if condition \eqref{eqexistprobreg} in Corollary \ref{corexist} is replaced by the more general one:

\begin{itemize}
\item[(H)]  There is  $\delta_0>0$  such that $f(t,x)<0$ for all $t\in [0,T]$ and all  $x\in (0,\delta_0)$ and
      $\lambda^* x - f(t,x) > 0$   for all $t\in [0,T]$ and all $x>0$,   where $\lambda^*=\left(\frac{\pi}{T}\right)^2{+}\,\frac{a^2}4$.
\end{itemize}

However, condition (H) is neither satisfied for the problem presented in Example \ref{ex} because in that case
$$\lambda^* x - f(t,x) \approx 10.5096x-149x^{0.99}+154x^{0.98}, $$
which is a sing-changing nonlinearity on the positive real line. Indeed, dividing by $x^{0.98}$ and making the change $y=x^{0.01}$, it is clear that $\lambda^* x - f(t,x)$ is negative on $(r_1,r_2)$ and positive on $(0,r_1)\cup (r_2,\infty)$, being $r_1\approx 103620.4135$ and $r_2\approx 3.7845*10^{111}$.

\end{remark}

\section{Appendix: The Green's function}

\subsection{The case $0<m<\displaystyle\frac a 2$}

The Green's function related to problem \eqref{eqprobhom} is given by
\begin{equation}\label{GF1}
  G_m(t,s)=\dfrac{1}{\lambda_2-\lambda_1}
  \begin{cases}
  \dfrac{e^{\lambda_2(t-s)}}{1-e^{\lambda_2T}}-\dfrac{e^{\lambda_1(t-s)}}{1-e^{\lambda_1T}},& 0\le s\le t\le T,\\
  \\
  \dfrac{e^{\lambda_2(t-s+T)}}{1-e^{\lambda_2T}}-\dfrac{e^{\lambda_1(t-s+T)}}{1-e^{\lambda_1T}},& 0\le
  t\le s\le T,
  \end{cases}
\end{equation}
where $\lambda_1$ and $\lambda_2$ are the characteristic roots of
the homogeneous equation
$$
x''(t)+ax'(t)+m^2x(t)=0,
$$
that is,
$$
\lambda_1=\dfrac{-a-\sqrt{a^2-4m^2}}{2}, \
\lambda_2=\dfrac{-a+\sqrt{a^2-4m^2}}{2}.
$$
Note that $\lambda_1<\lambda_2<0$.

\begin{remark}
  $G_m(s,s)$ is constant for $s\in [0,T]$,
  $$G_m(s,s)=\dfrac{1}{\lambda_2-\lambda_1}\left(\dfrac{1}{1-e^{\lambda_2T}}-\dfrac{1}{1-e^{\lambda_1T}}\right),$$
  and
  \begin{equation}\label{intGreen}
  \int_0^TG_m(t,s)ds=\frac{1}{\lambda_1\lambda_2}=\frac{1}{m^2}.
  \end{equation}
\end{remark}

The Green's function \eqref{GF1} satisfies
\begin{equation}\label{GF1-p1}
G_m(t,s)\ge G_m(s,s)>0,
\end{equation}
and
\begin{equation}\label{GF1-p2}
G_m(s,s)\ge c_m \,G_m(t,s),
\end{equation}
for all $(t,s)\in [0,T]\times[0,T]$, where
\begin{equation}\label{c1}
c_m=\frac{G_m(s,s)}{\max\{G_m(t,s):t,s\in [0,T]\}},
\end{equation}
that is,
\begin{equation}\label{c2}
  c_m=-\frac{\lambda_2}{\lambda_2-\lambda_1}
  \frac{e^{\lambda_2 T}-e^{\lambda_1 T}}{(1-e^{\lambda_2 T})\left(\frac{(1-e^{\lambda_2T})\lambda_1}{(1-e^{\lambda_1T})\lambda_2}\right)^{\frac{\lambda_1}{\lambda_2-\lambda_1}}}.
\end{equation}

\subsection{The case $m=\displaystyle\frac a 2$}

The Green's function related to the problem \eqref{eqprobhom} is given by
\begin{equation}\label{GF2}
  G_m(t,s)=\dfrac{1}{e^{mT}-1}
  \begin{cases}
  e^{m(t-s)}\left[\dfrac{Te^{mT}}{e^{mT}-1}+s-t\right],& 0\le s\le t\le T,\\
  \\
  e^{m(T-s+t)}\left[\dfrac{T}{e^{mT}-1}+s-t\right],& 0\le
  t\le s\le T.
  \end{cases}
\end{equation}
Then we have
\begin{equation*}
  \int_0^TG_m(t,s)ds=\frac{1}{m^2},
  \end{equation*}
\begin{equation*}
G_m(t,s)\ge G_m(s,s)>0,
\end{equation*}
and
\begin{equation*}
G_m(s,s)\ge c_m \,G_m(t,s),
\end{equation*}
for all $(t,s)\in [0,T]\times[0,T]$, where
\begin{equation*}
c_m=\frac{G_m(s,s)}{\max\{G_m(t,s):t,s\in [0,T]\}},
\end{equation*}
that is,
\begin{equation*}
c_m=\kappa e^{1-\kappa},
\end{equation*}
with
$$
\kappa=\frac{mT}{e^{mT}-1}.
$$

\subsection{ The case $m>\displaystyle\frac a 2$ and $m^2<\left(\dfrac{\pi}{T}\right)^2+\left(\dfrac{a}{2}\right)^2$}

We use the following notation:
$$\gamma=-\dfrac{a}{2}, \ \delta=\dfrac{\sqrt{4m^2-a^2}}{2}, \ D=1-2e^{\gamma T}\cos(\delta T)+e^{2\gamma T}.$$
The Green's function related to the problem \eqref{eqprobhom} is given by
\begin{equation}\label{GF3}
  G_m(t,s)=\dfrac{1}{\delta D}
  \begin{cases}
  e^{\gamma(T+t-s)}\sin(\delta(T+s-t))+e^{\gamma(t-s)}\sin(\delta(t-s)),& 0\le s\le t\le T,\\
  \\
  e^{\gamma(T+t-s)}\sin(\delta(T+t-s))+e^{\gamma(2T+t-s)}\sin(\delta(s-t)),& 0\le
  t\le s\le T.
  \end{cases}
\end{equation}
Then we have
\begin{equation*}
  \int_0^T G_m(t,s)ds=\frac{1}{m^2},
  \end{equation*}
\begin{equation*}
G_m(t,s)\ge G_m(s,s)>0,
\end{equation*}
and
\begin{equation*}
G_m(s,s)\ge c_m \,G_m(t,s),
\end{equation*}
for all $(t,s)\in [0,T]\times[0,T]$, where
\begin{equation*}
c_m=\frac{G_m(s,s)}{\max\{G_m(t,s):t,s\in [0,T]\}}=\frac{\dfrac{1}{\delta D}e^{\gamma T}\sin(\delta T)}{\max\{G_m(t,s):t,s\in [0,T]\}}.
\end{equation*}

If $a=0$ then $\gamma=0,\, \delta=m,\, D=2(1-\cos m T ) $ and it is known that $\max\{G_m(t,s):t,s\in [0,T]\}=\frac{1}{2m\sin\frac{mT}{2}}$, see \cite{Tor}. Therefore we have $c_m=\cos\frac{mT}{2}$.

On the other hand, if $a>0$ it is difficult to find exactly the maximum of $G_m(t,s)$, but we are able to get an upper bound in the following way: since $0<\delta T<\pi$ we have for $0\le s\le t\le T$ (the case $0\le t\le s\le T$ is analogous)
$$
\aligned
e^{\gamma(T+t-s)}\sin(\delta(T+s-t))&+e^{\gamma(t-s)}\sin(\delta(t-s))\le \sin(\delta(T+s-t))+\sin(\delta(t-s))\\
&=2\sin\frac{\delta T}{2}\cos\left(\delta(\frac{T}{2}-(t-s))\right)\le 2\sin\frac{\delta T}{2}.
\endaligned
$$
So, $\max\{G_m(t,s):t,s\in [0,T]\}\le \frac{2}{\delta D}\sin\frac{\delta T}{2}$ which yields $c_m\ge e^{\gamma T}\cos\frac{\delta T}{2}$.

\begin{remark}
  If $m^2=\left(\dfrac{\pi}{T}\right)^2+\left(\dfrac{a}{2}\right)^2$ then $G_m(s,s)=0$.
\end{remark}
\section*{Acknowledgements}
J. A. Cid was partially supported by Ministerio de Educaci\'on y Ciencia, Spain, and FEDER, Project MTM2010-15314, G. Infante was partially supported by G.N.A.M.P.A. - INdAM (Italy), M. Tvrd\'y was supported by GA \v{C}R Grant P201/14-06958S and  RVO: 67985840 and M. Zima was partially supported by the Centre for Innovation and Transfer of Natural Science and Engineering Knowledge of University of Rzesz\'ow.

\end{document}